\documentclass[a4paper,12pt]{article}
\usepackage{amssymb,amsmath,amsthm,latexsym}
\usepackage{amsfonts}
	\usepackage{amsfonts}
\usepackage{graphicx}
\usepackage[pdftex,bookmarks,colorlinks=false]{hyperref}
\usepackage{verbatim}
\usepackage{caption}
\usepackage{subcaption}
\usepackage[hmargin=1.15in,vmargin=1.15in]{geometry}

\usepackage{fancyhdr}
\newtheorem{theorem}{Theorem}[section]

\newtheorem{corollary}[theorem] {Corollary}
\newtheorem{definition}[theorem]{Definition}

\newtheorem{lemma} [theorem]{Lemma}
\newtheorem{observation}[theorem]{Observation}

\newtheorem{proposition}[theorem]{Proposition}

\title{\bf {\sc Weak Integer Additive Set-Indexed Graphs: A Creative Review}}

\author{K P Chithra$^1$, K A Germina$^2$ and N K Sudev$^3$ \\ \\
$^1$ {\small Naduvath mana, Nandikkara, Thrissur-680301, India.} \\  {\small  email:{\em chithrasudev@gmail.com}}\\
$^2$ {\small Department of Mathematics, School of Mathematical \& Physical Sciences}\\{\small Central University of Kerala, Kasaragod-671316, India.}\\ {\small email:{\em srgerminaka@gmail.com}}\\
$^3$ {\small Department of Mathematics, Vidya Academy of Science \& Technology}\\{\small Thalakkottukara, Thrissur-680501, India.}\\ {\small email:{\em sudevnk@gmail.com}}
}

\date{}

\begin{document}
\maketitle

\begin{abstract}
For a non-empty ground set $X$, finite or infinite, the {\em set-valuation} or {\em set-labeling} of a given graph $G$ is an injective function $f:V(G) \to \mathcal{P}(X)$, where $\mathcal{P}(X)$ is the power set of the set $X$. A set-indexer of a graph $G$ is an injective set-valued function $f:V(G) \to \mathcal{P}(X)$ such that the function $f^{\ast}:E(G)\to \mathcal{P}(X)-\{\emptyset\}$ defined by $f^{\ast}(uv) = f(u ){\ast} f(v)$ for every $uv{\in} E(G)$ is also injective., where $\ast$ is a binary operation on sets. An integer additive set-indexer (IASI) is defined as an injective function $f:V(G)\to \mathcal{P}({\mathbb{N}_0})$ such that the induced function $g_f:E(G) \to \mathcal{P}(\mathbb{N}_0)$ defined by $g_f (uv) = f(u)+ f(v)$ is also injective, where $\mathbb{N}_0$ is the set of all non-negative integers and $\mathcal{P}(\mathbb{N}_0)$ is its power set. A weak IASI is an IASI $f$ such that $|f^+(uv)|= \text{max}(f(u),f(v))$. In this paper, we critically and creatively review the concepts and properties of weak integer additive set-valued graphs.
\end{abstract}


\section{Preliminaries}

For all  terms and definitions, not defined specifically in this paper, we refer to \cite{FH}. Unless mentioned otherwise, all graphs considered here are simple, finite and have no isolated vertices.

The researches on graph labeling problems commenced with the introduction of the concept of number valuations of graphs in \cite{AR1}. Since then, the studies on graph labeling have contributed significantly to the researches in graph theory and associated filelds. Graph labeling problems have numerous theoretical and practical applications. Many types of graph labelings are surveyed and listed in \cite{JAG}.

Motivated from various problems related to social interactions and social networks, Acharya introduced, in \cite{BDA1}, the notion of set-valuation of graphs analogous to the number valuations of graphs. For a non-empty ground set $X$, finite or infinite, the {\em set-valuation} or {\em set-labeling} of a given graph $G$ is an injective function $f:V(G) \to \mathcal{P}(X)$, where $\mathcal{P}(X)$ is the power set of the set $X$. 

Also, Acharya defined a {\em set-indexer} of a graph $G$ as an injective set-valued function $f:V(G) \to \mathcal{P}(X)$ such that the function $f^{\ast}:E(G)\to \mathcal{P}(X)-\{\emptyset\}$ defined by $f^{\ast}(uv) = f(u ){\ast} f(v)$ for every $uv{\in} E(G)$ is also injective, where $\mathcal{P}(X)$ is the set of all subsets of $X$ and $\ast$ is a binary operation on sets. 

Taking the symmetric difference of two sets as the operation  between two set-labels of the vertices of $G$, the following theorem was proved in \cite{BDA1}.

\begin{theorem}\label{T-SIG}
\cite{BDA1} Every graph has a set-indexer.
\end{theorem}

During a personal communication with the second author, Acharya introduced the notion of integer additive set-indexers of graphs, using the concept of sum sets of two sets of non-negative integers and the definition was first appeared in \cite{GA} as follows. 

\begin{definition}\label{DEFN-1}{\rm
Let $\mathbb{N}_0$ be the set of all non-negative integers. An {\em integer additive set-labeling} (IASL, in short) is an injective function $f:V(G)\to \mathcal{P}(\mathbb{N}_0)$.

An {\em integer additive set-indexer} (IASI) is defined as an injective function $f:V(G)\to \mathcal{P}(\mathbb{N}_0)$ such that the induced function $g_f:E(G) \to \mathcal{P}(\mathbb{N}_0)$ defined by $f^+ (uv) = f(u)+ f(v)=\{a+b: a \in f(u), b \in f(v)\}$ is also injective. A graph $G$ which admits an IASI is called an integer additive set-indexed graph. 

An IASI is said to be {\em $k$-uniform} if $|f^+(e)| = k$ for all $e\in E(G)$. That is, a connected graph $G$ is said to have a $k$-uniform IASI if all of its edges have the same set-indexing number $k$.}
\end{definition}

If either $f(u)$ or $f(v)$ is countably infinite, then clearly their sum set will also be countably infinite and hence the study of the cardinality of the sum set $f(u)+f(v)$ becomes trivial. Hence, we restrict our discussion to finite sets of non-negative integers. We denote the cardinality of a set $A$ by $|A|$. 

The cardinality of the set-label of an element (vertex or edge) of a graph $G$ is called the {\em set-indexing number} of that element. If the set-labels of all vertices of $G$ have the same cardinality, then the vertex set $V(G)$ is said to be {\em uniformly set-indexed}. An element is said to be {\em mono-indexed} if its set-indexing number is $1$.

The bounds for the set-indexing numbers of edges of an IASI graph was estimated in the following lemma.

\begin{lemma}\label{L-IASI3}
\cite{GS1} If $f$ is an IASI of a graph $G$, then $\text{max}(|f(u)|, |f(v)|) \le |f^+(uv)|= |f(u)+f(v)| \le |f(u)| |f(v)|$, where $u,v \in V(G)$.
\end{lemma}

Further studies on integer additive set-indexers have been progressed on the basis of this lemma and the set-indexing numbers of the elements of given IASI graphs. In the following discussions, we report the properties and characteristics of a special type IASI graphs, the set-label of whose elements follow a definite pattern.

\section{WIASI of Graphs}

\subsection{Introduction}

The IASIs induces the minimum and maximum possible set-indexing numbers for the edges of given graphs are of special interest. Hence, we defined

\begin{definition}{\rm
\cite{GS1,GS2}A {\em weak integer additive set-indexer} (WIASI) of a graph $G$ is an IASI $f$ such that $|f^+(uv)| = max(|f(u)|,|f(v)|)$ for all $u, v \in V(G)$ and a {\em strong integer additive set-indexer} (SIASI) of $G$ is an IASI such that if $|g_f(uv)| = |f(u)|\,|f(v)|$ for all $u, v \in V(G)$.} 
\end{definition}

The second and third authors of this paper have studied the properties and characteristics of WIASIs and the requirements for certain graphs to be a WIASI graphs.

It was observed that if a graph $G$ admits a WIASI, then at least one end vertex of every edge of $G$ is mono-indexed. If two end vertices of an edge of $G$ are mono-indexed, then that edge will also be mono-indexed. The minimum number of mono-indexed edges possible in a WIASI graph is defined to be the {\em sparing number} of that grph. The {\em sparing number} of a graph $G$ is denoted by $\varphi(G)$.

An obvious result about an IASI $f$ of a given graph $G$ is the restriction of it to the substructures of $G$ which led us to the following result. A property is said to be {\em hereditary} if existence of that property for an object implies its existence for any subobject of that object.

\begin{proposition}\label{P-WIASI2}
\cite{GS0} If a given graph $G$ admits a WIASI $f$, then any subgraph $H$ of $G$ also admits a WIASI $f^*$, which is the restriction of $f$ to $V(H)$. That is, the existence of WIASI is a hereditary property.
\end{proposition}

The most important characteristic of a WIASI graph is the following.

\begin{theorem}\label{T-WIASI3}
\cite{GS3} A graph $G$ admits a WIASI if and only if $G$ is bipartite or $G$ has at least one mono-indexed.
\end{theorem}

The proof of this theorem follows from the fcts that  $G$ is a bipartite graph, then all vertices in one partition can be labeled by distinct singleton sets and all vertices in the other partition can be labeled by non-singleton sets and if $G$ is not bipartite, some of the mono-indexed are adjacent. Hence,

\begin{observation}\label{O-WIASI4}
\cite{GS3} The sparing number of a bipartite graph is $0$.
\end{observation}

Invoking Theorem \ref{T-WIASI3}, the possible number of mono-indexed edges in a cycle $C_n$ of length $n$.

\begin{observation}\label{O-WIASI5}
\cite{GS3} A cycle $C_n$ has odd number of mono-indexed edges if $n$ is odd and even number of mono-indexed edges if $n$ is even.
\end{observation}

An obvious result due to Theorem \ref{T-WIASI3} is

\begin{theorem}\label{T-WIASI6}
Every graph admits a weak integer additive set-indexer.
\end{theorem}

\subsection{Sparing Number of Graphs}

As all graphs are WIASI graphs, our interest is to determine the sparing number of a WIASI graph. We call the WIASI of a graph $G$, which yields the least possible number of mono-indexed edges, the {\em primary WIASI}.  Unless mentioned otherwise, the weak IASI of a graph, mentioned during the study about sparing number of various graphs and graph classes, is the primary weak IASI.

Due to Observation \ref{O-WIASI4}, the sparing number of paths, trees and even cycles is $0$. For odd cycles, if we label alternate vertices by distinct singleton vertices and distinct non-singleton vertices in such a way that no two adjacent vertices have non-singleton set-labels, then there exists exactly one mono-indexed egde. Hence,

\begin{observation}\label{O-WIASI7}
The sparing number of an odd cycle is $1$.
\end{observation}

Under a primary WIASI of a complete graph $K_n$, only one vertex of $K_n$ can be labeled by no-singleton set, as all vertices of $K_n$ are adjacent to each other. As a result, we have

\begin{theorem}\label{T-WIASI8}
\cite{GS3} The sparing number of a complete graph $K_n$ is $\frac{1}{2}(n-1)(n-2)$.
\end{theorem}

\begin{proposition}\label{P-WIASI8a}
\cite{GS6} The sparing number of a complete graph $K_n$ is equal to the number of triangles which contain the vertex that is not mono-indexed in $K_n$.
\end{proposition}

A solution to the question regarding the sparing number of an arbitrary graph $G$ is obtained from Theorem \ref{T-WIASI3} as 

\begin{theorem}\label{T-WIASI9}
\cite{GS5} Let $G$ be a WIASI graph and let $E'$ be a minimal subset of $E(G)$ such that $G-E'$ is bipartite. Then, $\varphi(G)=|E'|$. 
\end{theorem}

Invoking Theorem \ref{T-WIASI9}, it is obvious that the sparing number of a bipartite graph $G$ is the difference  between the sizes of $G$ and its maximal bipartite subgraph. Hence, we have

\begin{observation}\label{O-WIASI9a}
The number of edges in a maximal bipartite subgraph of a given graph $G$ is $|E(G)|-\varphi(G)$. 
\end{observation}

\subsection{Sparing Number and Some Graph Parameters}

The relation between the sparing number and certain other parameters such as chromatic number, matching number, covering number etc. are studied in \cite{GS5}. The major results proved in that studies are mentioned in this section.

\begin{theorem}\label{T-WIASI10}
\cite{GS5} Let $G$ be a $k$-chromatic graph. Then, its sparing number is the total number of vertices of $G$ which are not in the two maximal color classes of $G$.
\end{theorem}

That is, if $\mathcal{C}_1$ and $\mathcal{C}_2$ are the two maximal color classes of a $k$-chromatic graph $G$, then the sparing number of $G$ is the number of all vertices in $G-(\mathcal{C}_1 \cup \mathcal{C}_2)$. 

For certain types of graphs, another beautiful relation between the chromatic number and the sparing number exists, as stated in the following theorem.

\begin{theorem}\label{T-WIASI10a}
\cite{GS5} If $G$ is a bipartite graph or cycle, then $\chi(G)-\varphi(G)=2$, where $\chi(G)$ is the chromatic number of $G$.
\end{theorem}

The following theorem establishes the relation between the  sparing number and the matching number $\nu$ of paths and cycles.

\begin{theorem}\label{T-WIASI11}
\cite{GS5} If $G$ is a path or a cycle on $n$ vertices, then $\varphi(G)=\lceil \frac{n}{2} \rceil-\nu(G)$, where $\nu(G)$ is the matching number of $G$. 
\end{theorem}

The relation between the sparing number and the matching number of certain other graphs, which can be decomposed into a finite number of odd cycles, was established in the following theorem. 

\begin{theorem}\label{T-WIASI12}
\cite{GS5} If $G$ can be decomposed into finite number of odd cycles, then $\varphi(G)=\sum_{i=1}^m{\lceil{\frac{n_i}{2}}\rceil} - \nu(G)$, where $m$ is the number of edge-disjoint cycles in $G$ and $n_i$ is the size of the cycle $C_i$ in $G$.
\end{theorem}

Theorem \ref{T-WIASI12} does not holds for edge disjoint union of $n$ cycles containing more than one odd cycle, as the relation $\nu(G) =\sum_{i=}^{n}C_{n_i}$ does not hold for $G$.  As a result we proved the following result.

\begin{theorem}\label{WIASI12a}
\cite{GS5} If $G$ is an Eulerian graph that has at most one even cycle then $\varphi(G)=\sum_{i=1}^m{\lceil{\frac{n_i}{2}}\rceil} - \nu(G)$, where $m$ is the number of edge-disjoint cycles and $n_i$ is the size of the cycle $C_i$ in $G$.
\end{theorem}

We now prove the following relation between the sparing number and matching number of a graph  whose maximal bipartite subgraph is a complete bipartite graph.

\begin{theorem}\label{T-WIASI13}
Let $G$ be a WIASI graph on $n$ vertices and $m$ edges, whose maximal bipartite subgraph is a complete bipartite graph. Then, the sparing number of $G$, $\varphi(G)= m-\nu(n-\nu)$, where $\nu$ is the matching number of $G$.
\end{theorem}
\begin{proof}
Given that the maximal bipartite subgraph $H$ of a given graph $G$ is complete bipartite. Clearly, $V(H)=V(G)$. Let $(X, Y)$ be the bipartition of $V(H)$ such that $|X| \le |Y|$. Then, $\nu=\nu(G)=\nu(H)=|X|$. Therefore, $|Y|= n-\nu$. Since, $H$ is complete bipartite, $|E(H)|=\nu (n-\nu)$. By Theorem, \ref{T-WIASI9}, $\varphi(G)=|E(G)|-|E(H)| = m-\nu(n-\nu)$.
\end{proof}

The existence of the relations between the sparing number and the matching number of graphs other than those mentioned above, if exists, are yet to be verified. 

The relation between the number of mono-indexed vertices of a graph and its covering number is explained in the given problem.

\begin{theorem}\label{T-WIASI14}
\cite{GS5} The minimum number of mono-indexed vertices in a WIASI graph $G$ is equal to the covering number of $G$. 
\end{theorem}

The above theorem follows from the fact that, if $S$ is the minimal vertex cover of $G$, all vertices of $S$ can be labeled by singleton sets. Moreover, this fact yields,

\begin{theorem}\label{T-WIASI14a}
If $S$ is the minimal cover of a graph $G$ on $n$ vertices, then its sparing number $\varphi(G)= |E(G)|-\sum_{v\notin S}d(v)$.
\end{theorem}
\begin{proof}
Since $S$ is a covering of $G$, no edges of $G$ have their both end vertices in $V_S$. Therefore, all vertices in $V_S$ can be labeled by distinct non-singleton sets. Therefore, all the edges which are incident on any vertex in $V-S$ have non-singleton set-labels. That is, the number of edges in $G$ that are not mono-indexed is $\sum_{v\notin S}d(v)$. Therefore, $\varphi(G)= |E(G)|-\sum_{v\notin S}d(v)$.
\end{proof}

The above result can be related to the independence number of graphs as follows.

\begin{theorem}
\cite{GS5} If $G$ is a WIASI graph, then the maximum number of vertices that are not mono-indexed in $G$ is equal to the independence number of $G$.
\end{theorem}

The proof of the above theorem is an immediate consequence of the fact that the maximal independent set of a graph $G$ is the complement of the minimal vertex cover of $G$. Hence, the sparing number of $G$ can be related to the maximal independent set of $G$ as follows.

\begin{theorem}
If $I$ is the maximal independent set of a WIASI $G$, then its sparing number is $\varphi(G)= |E(G)|-\alpha(G). \bar{d_I}(v)$, where $\alpha(G)$ is the independence number of $G$ and $\bar{d_I}(V)$ is the average degree of vertices in $I$.
\end{theorem}
\begin{proof}
If $I$ is the maximal independent set of a graph $G$, then all vertices of $I$ can be labeled by distinct non-singleton sets. Then, no edge of $G$ that has an end vertex in $I$ will be mono-indexed. Therefore, its sparing number $\varphi(G)= |E(G)|-\sum_{v\in I}d(v)$.

Now, note that $\bar{d_I}(v)=\frac{\sum_{v\in I} (d(v)}{\alpha(G)}$. Therefore, $\varphi(G)= |E(G)|-\alpha(G).\bar{d_I}(v)$.
\end{proof}

\section{WIASIs of Graph Operations}

In this section, we review the admissibility of WIASIs of different graph operations of given WIASI graphs.

\subsection{WIASIs of Graph Unions}
The union of two graphs we mention here need not be the disjoint union. First, we discuss the admissibility of weak IASI by the union of two graphs $G_1$ and $G_2$.

\begin{theorem}\label{T-WISASI-WO-U1}
\cite{GS4} Let $G_1$ and $G_2$ be two WIASI graphs whose WIASIs are $f_1$ and $f_2$ respectively. Then, $G_1\cup G_2$ admits a WIASI $f$ defined by 
\begin{equation*}
f(v)=
\begin{cases}
f_1(v) ~~\text{\rm if}~~ v\in V(G_1)\\
f_2(v) ~~\text{\rm if}~~ v\in V(G_2).
\end{cases}
\end{equation*}
\end{theorem}

It is to be noted that, if $G_1$ and $G_2$ have some common elements, then $f_1=f_2$ for the common vertices and hence $f_1^+=f_2^+$ for common edges of $G_1$ and $G_2$.

Theorem \ref{T-WISASI-WO-U1} can be generalised for the finite union of WIASI graphs $G= \cup_{i=1}^{n}G_i$ by defining the WIASI $f$ for $G$, by $f(v)=f_i(v) ~\text{if}~ v\in G_i$.

In view of the above results, the following theorem is on the sparing number of graph unions.

\begin{theorem}\label{T-WIASI-WO-U2}
\cite{GS4} Let $G_1$ and $G_2$ be two weak IASI graphs. Then, $\varphi(G_1\cup G_2) = \varphi(G_1)+\varphi(G_2)-\varphi(G_1\cap G_2)$.
\end{theorem}

\begin{corollary}\label{C-WIASI-WO-U3}
If the two graphs $G_1$ and $G_2$ are edge disjoint, then $\varphi(G_1\cup G_2) = \varphi(G_1)+\varphi(G_2)$.
\end{corollary}

Invoking Corollary \ref{C-WIASI-WO-U3}, the following theorem for Eulerian graphs was established.

\begin{theorem}\label{T-WIASI-WO-U4}
\cite{GS8} The sparing number of an Eulerian graph $G$ is equal to the number of edge disjoint odd cycles in $G$.
\end{theorem}

The proof of the above theorem follows from the fact that an Eulerian graph can be decomposed into edge disjoint cycles and from Corollary \ref{C-WIASI-WO-U3}.

\subsection{WIASIs of Graph Joins}

The following theorem is about the admissibility of WIASIs by the join of two WIASI graphs. 

\begin{theorem}\label{T-WIASI-WO-J1}
\cite{GS4} Let $G_1$ and $G_2$ be two weak IASI graphs. Then, the graph $G_1+G_2$ is a weak IASI graph if and only if either $G_1$ or $G_2$ is a $1$-uniform IASI graph.
\end{theorem}

The proof of Theorem \ref{T-WIASI-WO-J1} follows from the fact that every vertex of $G_1$ is adjacent to every vertex of $G_2$ in $G_1+G_2$. In view of this theorem, the sparing number of certain graph joins are established  in \cite{GS4}.

\begin{proposition}
The sparing number of $K_m\times K_n$ is $\frac{1}{2}(m+n-1)(m+n-2)$
\end{proposition}

\begin{proposition}\label{P-WIASI-WO-J2}
\cite{GS4} The sparing number of $F_{1,n}=K_1+P_n$ is $\frac{n}{2}$, where $P_n$ is a path of length $n$. 
\end{proposition}

\begin{proposition}\label{P-WIASI-WO-J3}
\cite{GS4} The sparing number of a wheel graph $W_{n+1}=C_n+K_1$ is $\lfloor \frac{n+1}{2} \rfloor$.
\end{proposition}

Using Proposition \ref{P-WIASI-WO-J3},we prove the sparing number of a double wheel. 

\begin{proposition}\label{P-WIASI-WO-J4}
The sparing number of a double wheel graph $DW_{2n+1}=2C_n+K_1$ is 
\begin{equation*}
\varphi(DW_{2n+1}=)
\begin{cases}
n+1 & \text{ if~ $n$ is odd}\\
n & \text{ if~ $n$ is even}.
\end{cases}
\end{equation*}.
\end{proposition}
\begin{proof}
A double wheel is of the form $DW_{2n+1}=2C_n+K_1=(C_n\cup  C_n)+K_1$, which can be considered as the union of two isomorphic wheel graphs with the same hub (the vertex $v$ that is not in $C_n$ and is adjacent to all vertices of $C_n$). Therefore, as in the case of wheel graph, an IASI yields minimum number of mono-indexed edges if we label the hub vertex $v$ by singleton sets and the vertices of $C_n$ alternately by singleton and non-singleton sets. Therefore, by Proposition \ref{P-WIASI-WO-J3}, each wheel graph component $C_n+\{v\}$ has $\lfloor \frac{n+1}{2} \rfloor$ mono-indexed edges. Therefore, $\varphi(DW_{2n+1})=2.\lfloor \frac{n+1}{2} \rfloor$. This completes the proof.
\end{proof}

\begin{corollary}
If $G= m C_n+K_1$, then $\varphi(G)=m.\lfloor \frac{n+1}{2} \rfloor$. 
\end{corollary}

\subsection{WISASIS of Graph Complements}

A graph $G$ and its complement $\bar{G}$ have the same set of vertices and hence $G$ and $\bar{G}$ have the same set-labels for their corresponding vertices. The set-labels of the vertices in $V(G)$ under a WIASI of $G$ need not form a WIASI for the complement of $G$. A set-labeling of $V(G)$ that defines a WIASI for both the graphs $G$ and its complement $\bar{G}$ may be called a {\em concurrent set-labeling}. The set-labels of the vertex set of $G$ mentioned in this section are concurrent.

The following result discuss about the sparing number of the complement of a WIASI graph $G$.

\begin{proposition}\label{P-WIASI-WO-J5}
\cite{GS4} Let $G$ be a connected weak IASI graph on $n$ vertices. If its complement $\bar{G}$ is also a weak IASI graph, then $\bar{G}$ contains at least $\frac{1}{2}[(n-1)(n-2)-2r]$ mono-indexed edges, where $r=\Delta(G)$, the maximum vertex degree.
\end{proposition}

\begin{proposition}\label{P-WIASI-WO-J6}
\cite{GS4} If $G$ is a self-complementary $r$-regular graph on $n$ vertices which admits a weak IASI, then $G$ and $\bar{G}$ contain at least $\frac{1}{2}r(2r-1)$ mono-indexed edges. 
\end{proposition}


In this section, we report the admissibility of WIASI by certain products of WIASI graphs.

\subsection{WIASIs of Cartesian Product of WIASI graphs}

Let $G_1$ and $G_2$ be two graphs of order $n_1$ and $n_2$ respectively.The cartesian product $G_1\times G_2$ may be viewed as follows. Make $n_2$ copies of $G_1$, denoted by $G_{1i};~ 1\le i \le n_2$, where $G_{1i}$ corresponds to the vertex $v_i$ of $G_2$. Now, join the corresponding vertices of two copies $G_{1i}$ and $G_{1j}$ if the corresponding vertices $v_i$ and $v_j$ are adjacent in $G_2$. Thus, we view the product $G_1\times G_2$ as a union of $n_2$ copies of $G_1$ and a finite number of edges connecting two copies $G_{1i}$ and $G_{1j}$ of $G_1$ according to the adjacency of the corresponding vertices $v_i$ and $v_j$ in $G_2$, where $1\le i\neq j\le n_2$.

The WIASI of $G_1\times G_2$ is defined by assigning integral multiples of the set-labels of $G_1$ to the copies $G_{1i}$ of $G_1$ for odd $i$ and re-labeling the vertices of $G_{1i}$  corresponding to the non-adjacent mono-indexed vertices of $G_1$ by non-singleton sets and vice versa, for even $i$. Then, we have

\begin{theorem}\label{T-WIASI-GPCP1}
\cite{GS7}The cartesian product of two WIASI graphs admits a WIASI.
\end{theorem}

If a graph $G$ is the cartesian product of two graphs $G_1$ and $G_2$, then $G_1$ and $G_2$ are called the {\em factors} of $G$. A graph is said to be {\em prime} with respect to a given graph product if it is nontrivial and cannot be represented as the product of two non trivial graphs. 

\begin{theorem}\label{T-WIASI-GPCP2}
\cite{GS7} Let $G$ is a non-prime graph which admits a WIASI. Then, every factor of $G$ also admits a weak IASI.
\end{theorem}





\subsection{Corona of Weak IASI Graphs}

By {\em corona} of two graphs $G_1$ and $G_2$, denoted by $G_1\odot G_2$, is the graph obtained taking one copy of $G_1$ (which has $p_1$ vertices) and $p_1$ copies of $G_2$ and then joining the $i$-th point of $G_1$ to every point in the $i$-th copy of $G_2$. The number of vertices and edges in $G_1\odot G_2$ are $p_1(1+p_2)$ and $q_1+p_1q_2+p_1p_2$ respectively, where $p_i$ and $q_i$ are the number of vertices and edges of the graph $G_i, i=1,2$.

The following theorem reports the admissibility of WIASI by coronas of two graphs.

\begin{theorem}\label{T-WIASI-GP-C4}
\cite{GS7} Let $G_1$ and $G_2$ be two graphs on $m$ and $n$ vertices respectively. Then, $G_1\odot G_2$ admits a WIASI if and only if either $G_1$ is $1$-uniform or it has $r$ copies of $G_2$ that are $1$-uniform, where $r$ is the number of vertices in $G_1$ that are not mono-indexed.
\end{theorem}






\section{WIASIs of Some Associated Graphs}

In this section, we discuss the admissibility of WIASI by certain associated graphs of a given WIASI graph. Line graphs, total graphs, subdivisions, super subdivisions, minors etc. are some of the associated graphs of a given graph.

We consider two types of set-labelings for associated graphs. In the first type, the elements of an associated graph are labeled by the same set-labels of the corresponding elements of the original graph. If an associated graph are obtained by replacing some elements of $G$ by some new elements (which are not in $G$), then the set-labels of the new elements are the same set-labels of the replaced elements of $G$. This IASI is called an {\em induced IASI}. In the second type, each associated graph can be treated as an individual graph and are set-labeled independently without considering the set-labels of the given graph $G$. 

The following results explains whether certain associated graphs admit induced WIASIs.

\begin{proposition}\label{P-WIASI-AG1}
\cite{GS0} $G$ be a WIASI graph. Let a vertex $v$ of $G$ with $d(v)=2$, which is not in a triangle, and $u$ and $w$ be the adjacent vertices of $v$. Then, the graph $H=(G-v)\cup \{uw\}$ admits an induced WIASI if and only if $v$ is not mono-indexed or at least one of the edges $uv$ or $vw$ is mono-indexed. 
\end{proposition}

\begin{theorem}\label{T-WIASI-AG2}
The line graph $L(G)$ of a WIASI graph $G$ admits an induced WIASI if and only if at least one edge in each pair of adjacent edges in $G$, is mono-indexed in $G$. 
\end{theorem}

We know that the adjacency of two vertices in the line graph $L(G)$ corresponds to the adjacency of corresponding edges in the graph $G$. Hence, the proof of this theorem is an immediate consequence of the fact that a vertex in $L(G)$ is mono-indexed if the corresponding edge in $G$ is mono-indexed.

\begin{theorem}\label{T-WIASI-AG3}
The total graph $T(G)$ of a weak IASI graph $G$ admits an induced WIASI if and only if the WIASI defined on $G$ is $1$-uniform.
\end{theorem}

The proof follows from the fact that if $G$ is not $1$-uniform, then the adjacency of vertices in $T(G)$ corresponding to an edge and its end vertex that are not mono-indexed in $G$ makes a contradicts the definition of a WIASI. 

A {\em subdivision} of a graph $G$ is the graph obtained by introducing a new vertex to some or all edges of $G$. If a new vertex is introduced to all edges of $G$, then the new graph is said to be a {\em complete subdivision of $G$}. Therefore,

\begin{theorem}\label{T-WIASI-AG4}
The complete subdivision of a WIASI graph $G$ admits an induced WIASI  if and only if $G$ is $1$-uniform.
\end{theorem}
\begin{proof}
Let $e=uv$ be an edge of $G$. If $e$ is not mono-indexed, then either $u$ or $v$ is not mono-indexed. Without loss of generality, let $v$ is not mono-indexed. Now, introduce a new vertex $w$ to the edge $uv$. Then, the edge $uv$ will be replaced by two edges $uw$ and $vw$. If we label $w$ by the same set-label of  $e$, then the edge $wv$ has both its end vertices not mono-indexed. Therefore, the IASI is not a WIASI.
\end{proof}

\section{WIASIs of Some Graphs and Graph Powers}

In this section, we mention the sparing number of the graphs with some structural properties and their powers. For all graph classes mentioned here, refer to \cite{BLS}, \cite{BHLL} and \cite{GCO}.

The $r$-th power of a simple graph $G$ is the graph $G^r$ whose vertex set is $V$, two distinct vertices being adjacent in $G^r$ if and only if their distance in $G$ is at most $r$.
First, note the following theorem on graph powers.

\begin{theorem}\label{T-Gdiam}
\cite{EWW} Let $G$ be a graph with diameter $d$. Then, $G^d$ is a complete graph.
\end{theorem}

The following theorems explains about the sparing number of arbitrary powers of paths and cycles.

\begin{theorem}\label{T-WIAISI-GT-P1}
\cite{GS9} Let $P_{n-1}$ be a path graph on $n$ vertices. Then, its spring number is  $\frac{r-1}{2(r+1)}[r(2n-1-r)+2i]$.
\end{theorem}

\begin{theorem}\label{T-WIAISI-GT-P2}
Let $C_n$ be a cycle on $n$ vertices and let $r$ be a positive integer less than $\lfloor \frac{n}{2} \rfloor $. Then the sparing number of the the $r$-th power of $C_n$ is given by $\varphi(C_n^r)=\frac{r}{r+1}((r-1)n+2i) ~~\text{if} ~ n\equiv i~(mod~(r+1))$.
\end{theorem}

Invoking Theorem \ref{T-Gdiam}, we verify the existence of weak IASI for complete bipartite graphs.

\begin{theorem}\label{T-WIAISI-GT-P3}
The sparing number of the square of a complete bipartite graph $K_{m,n}$ is $\frac{1}{2}(m+n-1)(m+n-1)$.
\end{theorem}

First we consider a sun graph, defined in \cite{BLS}. 

An {\em $n$-sun} or a {\em trampoline} is a chordal graph on $2n$ vertices, where $n\ge 3$, whose vertex set can be partitioned into two sets $U = \{u_1,u_2,c_3,\ldots, u_n\}$ and $W = \{w_1,w_2,w_3,\ldots, w_n\}$ such that $W$ is an independent set of $G$ and $w_j$ is adjacent to $u_i$ if and only if $j=i$ or $j=i+1~(mod ~ n)$. {\em A complete sun} is a sun $G$ where the induced subgraph $\langle U \rangle$ is complete.

\begin{theorem}\label{T-WIASI-GC1}
\cite{GS6} The sparing number of a complete sun graph $S_{n}$ is $\frac{1}{2}n(n-1)$.
\end{theorem}

The sparing number of a sun graph, that is not complete, is yet to be estimated because of the uncertainty in the adjacency pattern between $U$ and $W$. The sparing number of the square of a complete sun graph is given in the following theorem.

\begin{theorem}\label{T-WIASI-GC-P1}
\cite{GS9} Let $G$ be the complete sun graph on $2n$ vertices. Then, the sparing number of $G^2$ is 
\begin{equation*}
\varphi(G^2)=
\begin{cases}
n^2+1 & \mbox{if ~ $n$ is odd}\\
\frac{n}{2}(2n-1) & \mbox{if ~ $n$ is even}.
\end{cases}
\end{equation*}
\end{theorem} 

We now extend Theorem \ref{T-WIASI-GC-P1} to arbitrary powers of a complete sun graph.

\begin{theorem}\label{T-WIASI-GC-P1g}
For $n\ge 4$, the sparing number of an arbitrary power of the complete sun graph $G$ is $\varphi(G^r)=(n-1)(2n-1)$,where $r\ge 3$.
\end{theorem}
\begin{proof}
For $1 \le i \le n$, a vertex $w_i$ in $W$ is at a distance at most $2$ from all the vertices of $V$ and adjacent to the vertices $w_{i-1}$ and $w_{i+1}$ of $W$ and at a distance $3$ from all other vertices of $W$. That is, the diameter of $G$ is $3$. Therefore, by Theorem \ref{T-Gdiam}, $G^r$, for all $r\ge 3$,  is a complete graph on $2n$ vertices. Hence, by Theorem \ref{T-WIASI8}, the sparing number of $G^r$ is $\varphi(G^r)=\frac{1}{2}(2n-1)(2n-2)=(n-1)(2n-1)$.
\end{proof}

A {\em split graph} is defined in \cite{GCO} as a graph in which the vertices can be partitioned into a clique $K_r$ and an independent set $S$. A split graph is said to be a {\em complete split graph} if every vertex of the independent set $S$ is adjacent to every vertex of the the clique $K_r$ and is denoted by $K_S(r,s)$, where $r$ and $s$ are the orders of $K_r$ and $S$ respectively.

The sparing number of a split graph is given by

\begin{theorem}\label{T-WIAISI-GC-P2}
\cite{GS6} The split graph sparing number of $G$ is equal to the number of triangles in $G$ containing the vertex that is not mono-indexed in its clique. 
\end{theorem}

The following theorem estimates the sparing number of a complete split graph $G$.

\begin{theorem}\label{T-WIAISI-GC-P3}
\cite{GS6} The sparing number of a complete split graph $K_S(r,s)$ is $\frac{1}{2}r(r-1)$. 
\end{theorem}

It can be noted that the sparing number of a complete split graph is equal to the sparing number of the maximal clique in it.

\begin{theorem}\label{T-WIAISI-GC-P4}
\cite{GS9} Let $G=K_S(r,s)$ be a complete split graph without isolated vertices. Then, the sparing number of $G^2$ is $\frac{1}{2}[(r+s-1)(r+s-2)]$, where $r=|V(K_r)|$ and $s=|S|$.
\end{theorem}

The diameter of a split graph is $3$ and hence the cube and higher powers of $G$ are complete graphs of order $r+s$. Hence, we have

\begin{theorem}\label{T-WIAISI-GC-P5}
The sparing number of an arbitrary power of a split graph $G$ is $\varphi(G^r)=\frac{1}{2}[(r+s-1)(r+s-2)]$.
\end{theorem}

A {\em windmill graph}, denoted by $W(n,k)$, is defined in \cite{JAG}, as an undirected graph constructed for $n\ge 2$ and $r\ge 2$ by joining $k$ copies of the complete graph $K_n$ at a shared vertex.

\begin{theorem}\label{T-WIAISI-GC-P6}
\cite{GS6} The sparing number of a windmill graph $W(n,k)$ is $\frac{k}{2}(n-1)(n-2)$.
\end{theorem}

 The following theorem establishes the sparing number of an arbitrary power of a windmill graph.

\begin{theorem}\label{T-WIAISI-GC-P7}
The sparing number of an arbitrary power of a windmill graph $G=W(n,k)$ is $\varphi(G^{r}) =\frac{1}{2}(k(n-1)(k(n-1)-1)$, for $r\ge 2$.
\end{theorem}
\begin{proof}
Since the windmill graph $G$ is the graph obtained by joining $k$ copies of $K_n$ with one common vertex, every pair of vertices in $G$ are at a distance $2$. That is, the diameter of a windmill graph is $2$. Therefore, $G^2$ is a complete graph on $n(k-1)+1$. Therefore, by Theorem \ref{T-WIASI8}, $\varphi(G^{r}) =\frac{1}{2}k(n-1)(k(n-1)-1)$, for $r\ge 2$.
\end{proof}

Now, we proceed to the theorem on the sparing number of squares of wheel graphs. 

\begin{proposition}\label{T-WIAISI-GC-P8}
The sparing number of the square of a wheel graph on $n+1$ vertices is $\frac{1}{2}n(n-1)$.
\end{proposition}

The proof of the above theorem follows from the fact that the diameter of wheel graphs is $2$. Then Theorem holds for any power $W_{n+1}^r$ for $r\ge 2$.

The {\em bipartite wheel graph} or a {\em gear graph} is defined in \cite{BLS} as a graph with a vertex added between each pair of adjacent graph vertices of the outer cycle of a wheel graph. The number of vertices in a gear graph $G=BW_n$ is $2n+1$.  The gear graph is bipartite and hence its sparing number is $0$. 

\begin{theorem}\label{T-WIAISI-GC-P9}
The sparing number of the $r$-th  power of a gear graph $BW_n$ is 
\begin{equation*}
\varphi(BW_n^r)=
\begin{cases}
\frac{1}{2}[(n+1)^2+5] & ~~\text{if}~~ r=2, n ~\text{is even}\\
\frac{1}{2}[(n+1)^2] & ~~\text{if}~~ r=2, n ~\text{is odd}\\
 n(n+1) & ~~\text{if}~~ r=3, n ~\text{is even}\\
 \frac{1}{2}(2n^2+3n+3) & ~~\text{if}~~ r=3, n ~\text{is odd}\\
n(2n-1) & ~~\text{if}~~ r=4.
\end{cases}
\end{equation*}
\end{theorem}
\begin{proof}
Let $G=BW_n$, where $u$ be the central vertex of $G$, $A=\{v_1,v_2,\ldots v_n\}$ be the set of vertices of $G$ which are adjacent to $u$ and $B=\{w_1,w_2,w_3, \ldots, w_n\}$ be the vertices of $G$ which are not adjacent to $u$.

\noindent {\em Case 1:} If $r=2$.

In $G^2$, all vertices in both $A$ and $B$ are adjacent to $u$ and all edges in $A$ are adjacent to each other. A vertex $v_i$ is adjacent to all other $n-1$ vertices in $A$, to $u$ and to two vertices $w_{i-1}$ and $w_{i+1}$ of $B$. Therefore, in $G^2$, $d(v_i)=n+2$ for all $v_i$ in $A$.  The vertex $w_i$ is adjacent to the vertices $w_{i+1}$ and $w_{i-1}$ of $B$ and to the vertices $v_i$ and $v_{i+1}$ of $A$. Therefore, $d(w_i)=5$ in $G^2$. Therefore, the number of edges in $G^2$ is $|E|=\frac{1}{2}\sum_{v\in V(G)}d(v)=\frac{1}{2}[n(n+2)+5n+2n]=\frac{1}{2}n(n+9)$.

If we label $u$ by a non-singleton set, then no other vertex in $G$ can be labeled by a non-singleton set. Therefore, the number of edges that are not mono-indexed, in this case is $2n$. If we label $u$ by a singleton set, then exactly one vertex of $A$, say and $\lfloor \frac{n-1}{2} \rfloor$ vertices, of $B$, except $w_1$ and $w_n$, can be labeled by distinct non-singleton sets. Therefore, the number mono-indexed edges, in this case, is $n+2+5\lfloor \frac{n-1}{2} \rfloor$, which is greater than $2n$ for all $n>1$. Hence, we adopt the second option. Then, the minimum number of mono-indexed edges in $G$ is $\frac{1}{2}n(n+9)-(n+2+5\lfloor \frac{n-1}{2} \rfloor)$. Here we have two cases.

\noindent {\em Subase 1.1:} If n is even, then the number of mono-indexed edges is $\frac{1}{2}n(n+9)-(n+2+5\frac{n-2}{2})=\frac{1}{2}[(n+1)^2+5]$.

\noindent {\em Subcase 1.2:} If n is odd, then the number of mono-indexed edges is $\frac{1}{2}n(n+9)-(n+2+5\frac{(n-1)}{2})=\frac{1}{2}[(n+1)^2]$.

\noindent {\em Case 2:} If $r=3$.

In $G^3$, every vertex of the set $A$ is adjacent to every other vertices in $V(G)$. Hence, $d(v)=2n$ for all $v\in A$. also, $d(u)=2n$. Now, every vertex $w_i$ is adjacent to $u$, every vertex of the set $A$ and to the vertices $w_{i+1}$ and $w_{i-1}$ of $B$. Therefore, $d(v)=n+3$ for all $v\in B$. Then, $|E|=\frac{1}{2}[2n+2n^2+n(n+3)]=\frac{1}{2}n(3n+5)$.

Since every vertex of $A\cup \{u\}$ is adjacent to all other vertices in $G$, only one vertex in this set is eligible for a non-singleton set-label. If we are to label the possible vertices of $B$ by non-singleton set-labels, then $\lfloor \frac{n}{2} \rfloor$ vertices are eligible for such labeling. Here we adopt the second option. Therefore, the number of edges in $G^3$, that are not mono-indexed, is $\lfloor \frac{n}{2} \rfloor(n+3)$. Then, the number of mono-indexed edges in $G^3$ is $\frac{1}{2}n(3n+5)-\lfloor \frac{n}{2} \rfloor(n+3)$. Here we have two cases.

\noindent {\em Subase 2.1:} If n is even, then the number of mono-indexed edges is $\frac{1}{2}n(3n+5)-\frac{1}{2}n(n+3)=n(n+1)$.

\noindent {\em Subcase 2.2:} If n is odd, then the number of mono-indexed edges is $\frac{1}{2}n(3n+5)-\frac{1}{2}(n-1)(n+3)=\frac{1}{2}(2n^2+3n+3)$.

\noindent {\em Case 3:} If $r\ge 4$.

The distance between any pair of vertices in a gear graph $G$ is a at most $4$. Therefore, for $r\ge 4$, $G^r$ is a complete graph. Hence, by Theorem \ref{T-WIASI8}, the sparing number of $G^r$ is $n(2n-1)$.
\end{proof}

\section{Weakly Uniform IASI Graphs}

A WIASI $f$ is said to be {\em $k$-uniform} if $f^+(e)=k$ for all edges $e$ of $G$. In this context, we say that $G$ admits a {\em $k$-uniform WIASI} ($k$-UWIASI, in short) or a {\em weakly $k$-uniform IASI}. If graph $G$ admits a weakly $k$-uniform IASI , then one end vertex of every edge of $G$ must be mono-indexed and the other end vertex must be set-labeled by distinct non-singleton sets of cardinality $k$ , for a unique positive integer $k>1$. If $G$ has $k$-uniform IASI for all values of $k$, then $G$ is said to have an {\em arbitrarily $k$-uniform IASI} or simply an {\em arbitrary Uniform IASI}. 

The following theorem characterises $k$-uniform WIASI graphs.

\begin{theorem}\label{T-WIASI-WU1}
\cite{GS1} For any positive integer $k>1$, a graph $G$ admits a weakly $k$-uniform IASI if and only if $G$ is bipartite.
\end{theorem}

\begin{proposition}\label{P-WIASI-WU1a}
A graph $G$ admits a UWIASI if and only if its sparing number is $0$.
\end{proposition}

The proof of this theorem is obvious as a consequence of Theorem \ref{T-WIASI-WU1} and Observation \ref{O-WIASI4}.

If $G$ admits a uniform WIASI, then one partition of its vertex can be labeled by non-singleton sets of cardinality $k$, for any positive value of $k$. Therefore, we have

\begin{proposition}\label{P-WIASI-WU3}
\cite{GS1} If a graph $G$ admits a uniform WIASI, then it has an arbitrary uniform IASI.  
\end{proposition}

The hereditary property of uniform WIASI, like the WIASI of any graph structure, is established in the following theorem.

\begin{theorem}\label{T-WIASI-WU4}
\cite{GS1} Any subgraph of a uniform WIASI also admit a uniform WIASI. That is, the existence of uniform WIASI is hereditary.
\end{theorem}

Equivalent to Theorem \ref{T-WIASI-WU4}, we can say that no supergraph of a non-uniform WIASI does not admit a uniform WIASI. 

We now proceed to check whether the graph operations and products of uniform WIASIs admit uniform WIASIs.

\begin{theorem}\label{T-WIASI-WU5}
The disjoint union of two uniform WIASI graphs also admits a uniform WIASI.
\end{theorem}
\begin{proof}
The proof similar to the proof of Theorem \ref{T-WISASI-WO-U1} given in \cite{GS4}. If $f_1$ and $f_2$ are the UWIASIs of the graphs $G_1$ and $G_2$ respectively, then the function $f$ defined on $G_1\cup G_2$, by
\begin{equation*}
f(v)=
\begin{cases}
f_1(v) & ~~\text{if}~~ v\in V(G_1)\\
f_2(v) & ~~\text{if}~~ v\in V(G_2).
\end{cases}
\end{equation*}
Clearly, $f$ is a UWIASI on $G=G_1\cup G_2$.
\end{proof}

In view of Theorem \ref{T-WIASI-WU5}, the sparing number of the disjoint union of two UWIASI graphs is $0$. It can be noticed that the union (need not be the disjoint union) of two UWIASI graphs, admits a UWIASI if, among the common vertices of $G_1$ and $G_2$, no two vertices in the same partition of $G_1$ are adjacent in $G_2$ and vice versa.

Next, we proceed to determine the sparing number of the join of two UWIASI graphs.

\begin{theorem}\label{T-WIASI-WU6}
Let the (bipartite) graphs $G_1(X_1,Y_1,E_1)$ and $G_2(X_2,Y_2,E_2)$ admit UWIASI, where $|X_i|=m_i$, $|Y_i|=n_i$ and  $|E(G_i)|=q_i$ for $i=1,2$. Then, the sparing number of $G_1+G_2$ is

\begin{equation*}
\varphi(G_1+G_2)=
\begin{cases}
q_1+m_2(m_1+n_1) & ~~\text{if}~~ m_1+n_1 \le m_2+n_2, m_2\le n_2\\
q_1+n_2(m_1+n_1) & ~~\text{if}~~ m_1+n_1 \le m_2+n_2, m_2\ge n_2\\
q_2+m_1(m_2+n_2) & ~~\text{if}~~ m_1+n_1 \ge m_2+n_2, m_1\le n_1\\
q_2+n_1(m_2+n_2) & ~~\text{if}~~ m_1+n_1 \ge m_2+n_2, m_1\ge n_1.
\end{cases}
\end{equation*}
\end{theorem}
\begin{proof}
Let $G_1+G_2$ is not a bipartite graph and hence by Theorem \ref{T-WIASI-WU1}, $G_1+G_2$ does not admit a UWIASI and contains some mono-indexed edges. By Theorem \ref{T-WIASI-WO-J1}, either $G_1$ or $G_2$ must be $1$-uniform. 

\noindent {\em Case 1:} Let $G_1$ has the minimum number of vertices. That is, let $m_1+n_1 \le m_2+n_2$. Since we require the minimum number of mono-indexed edges, label all the vertices of $G_1$ by distinct singleton sets. Then all $q_1$ edges of $G_1$ are mono-indexed. The vertices of one partition in $G_2$ can be labeled non-singleton sets. 

\noindent {\em Subcase 1.1:} If $m_2\le n_2$, then label the vertices in $X_2$ by singleton sets. Then, the mono-indexed edges of $G_1+G_2$ are the edges of $G_1$ and the edges between $X_2$ and $V(G_1)$. Therefore, the number of mono-indexed edges thus obtained is the minimum. Hence, $\varphi(G_1+G_2)=q_1+m_2(m_1+n_1)$.

\noindent {\em Subcase 1.2:} If $m_2\ge n_2$, then label the vertices in $Y_2$ by singleton sets so that we get minimum number of mono-indexed edges between $G_1$ and $G_2$. Then, the mono-indexed edges of $G_1+G_2$ are the edges of $G_1$ and the edges between $Y_2$ and $V(G_1)$. Hence, $\varphi(G_1+G_2)=q_1+n_2(m_1+n_1)$.

\noindent {\em Case 2:} If $G_2$ has the minimum number of vertices, (that is, if $m_1+n_1 \le m_2+n_2$), label all the vertices of $G_2$ by distinct singleton sets. Then, all $q_2$ edges of $G_1$ are mono-indexed. The vertices of one partition in $G_1$ can be labeled non-singleton sets. 

\noindent {\em Subcase 2.1:} If $m_1\le n_1$, then label the vertices in $X_1$ by singleton sets. Then, the mono-indexed edges of $G_1+G_2$ are the edges of $G_2$ and the edges between $X_1$ and $V(G_2)$. Therefore, $\varphi(G_1+G_2)=q_2+m_1(m_2+n_2)$.

\noindent {\em Subcase 2.2:} If $m_1\ge n_1$, then label the vertices in $Y_1$ by singleton sets. Then, the mono-indexed edges of $G_1+G_2$ are the edges of $G_2$ and the edges between $Y_1$ and $V(G_2)$. Therefore, $\varphi(G_1+G_2)=q_2+n_1(m_2+n_2)$.
\end{proof}

Next, we go through admissibility of UWIASI by the cartesian product of two UWIASI graphs and estimates its sparing number.

\begin{proposition}\label{T-WIASI-WU7}
The cartesian product of two UWIASI graphs also admits a UWIASI and its sparing number is $0$.
\end{proposition}
\begin{proof}
Let $G_1$ and $G_2$ be two UWIASI graphs. Then , by Theorem \ref{T-WIASI-WU1}, both $G_1$ and $G_2$ are bipartite. Since the cartesian product of two bipartite graphs is bipartite, $G_1\times G_2$ is also bipartite. Therefore, by Theorem \ref{T-WIASI-WU1}, $G_1\times G_2$ admits a UWIASI. Hence, by Proposition \ref{P-WIASI-WU1a}, the sparing number of $G_1\times G_2$ is $0$.
\end{proof}

The following theorems establish the admissibility of UWIASIs by certain associated graphs of WIASI graphs.

The following theorem is on the subdivisions of WIASI graphs.

\begin{theorem}\label{T-WIASI-WU8}
The complete subdivision of a (WIASI) graph admits a uniform WIASI.
\end{theorem}
The proof of this theorem follows from the fact that the complete subdivision of any graph is a bipartite graph.

A {\em super subdivision} graph of a graph $G$ is the graph obtained by replacing every edge of $G$ by a complete bipartite graph $K_{2,m}$. It is clear, from its definition,  that a super subdivision of a graph is bipartite. Therefore, we have

\begin{theorem}\label{T-WIASI-WU9}
The super subdivision of a (WIASI) graph admits a uniform WIASI.
\end{theorem}

\begin{theorem}\label{T-WIASI-WU10}
The graph obtained by contracting an edge of a UWIASI graph does not admit a UWIASI.
\end{theorem}

The graph obtained by contracting an edge of a bipartite graph will not be a bipartite graph and hence the proof is evident.

\section{Scope for Further Studies}
So far, we have reviewed the studies about weak integer additive set-indexers of certain graphs and their sparing numbers. We also propose some new results in this paper. Certain problems regarding the admissibility of weak IASI by various other graph classes, graph operations and graph products and finding the corresponding sparing numbers are still open.

More properties and characteristics of different types of IASIs, both uniform and non-uniform, are yet to be investigated. There are some more open problems regarding the necessary and sufficient conditions for various graphs and graph classes to admit certain IASIs.

\section*{Acknowledgement}
The authors dedicate this work to the memory of Professor Belamannu Devadas Acharya who introduced the concept of integer additive set-indexers of graphs.

\end{document}